\documentclass[11pt]{amsart}
\pdfoutput=1

\usepackage{amssymb}
\usepackage{amsthm}
\usepackage{amsfonts}
\usepackage{amsmath}
\usepackage{pmboxdraw}
\usepackage{verbatim} 
\usepackage{graphicx}
\usepackage{color}

\theoremstyle{plain}
\newtheorem*{thm*}{Theorem}
\newtheorem{thm}{Theorem}[section]
\newtheorem*{lem*}{Lemma}
\newtheorem{lem}[thm]{Lemma}

\theoremstyle{definition}
\newtheorem*{eg*}{Example}
\newtheorem*{egs*}{Examples}
\newtheorem*{def*}{Definition}

\theoremstyle{remark}
\newtheorem*{rmk*}{Remark}
\newtheorem*{rmks*}{Remarks}

\def\pa{\partial}

\def\Re{ \mathrm{Re}}

\def\erfc{\mathrm{erfc}}

\def\C{\mathbb{C}}

\def\E{\mathbf{E}}
\def\P{\mathbf{P}}
\def\R{\mathbb{R}}
\def\fR{\mathbf{R}}

\def\L{\mathbb{L}}

\def\fR{\mathbf{R}}

\newcommand{\1}{{\mathbf{1}}}

\begin{document}
\title[]{Edge scaling limit of the spectral radius for random normal matrix ensembles at hard edge}

\author{Seong-Mi Seo}
\address{Seong-Mi Seo\\
School of Mathematics\\
Korea Institute for Advanced Study\\
85 Hoegiro\\
Dongdaemun-gu\\
Seoul 02455\\
Republic of Korea}
\email{seongmi@kias.re.kr}




\begin{abstract}
We investigate a random normal matrix model with eigenvalues forced to be in the droplet, the support of the equilibrium measure associated with an external field. For radially symmetric external fields, we show that the fluctuations of the spectral radius around a hard edge tend to follow an exponential distribution as the number of eigenvalues tends to infinity. As a corollary, we obtain the order statistics of the moduli of eigenvalues.

\end{abstract}

\maketitle

\section[]{Introduction and results}

In random matrix theory, there have been numerous studies of the spectral radius of large size matrices. The limiting distributions of the largest eigenvalue of classical random matrix ensembles, Gaussian orthogonal, unitary, and symplectic ensembles, were studied by Forrester \cite{MR1236195} and Tracy-Widom \cite{MR1257246, MR1385083}. More generally, a type of universality for Wigner random matrices was proved by Soshnikov \cite{MR1727234}. The study of the scaling limit of correlation functions at the edge of the spectrum led the universality of the largest eigenvalue distribution for some invariant ensembles \cite{MR2306224}. 

The edge behavior of the spectrum of a random normal matrix is different from that of a random hermitian matrix, which is expressed in terms of Painlev\'{e} II. 
In the random normal matrix model, one considers random normal matrices of size $n$ with a probability measure of the form
\begin{align}\label{ProbM}
	d\mathcal{P}_n(M)=\frac{1}{\mathcal{Z}_n} e^{-n \mathrm{tr} Q(M)} dM,
\end{align}
where $Q:\C \to \R\cup \{+\infty\}$ is an external potential,
$dM$ is the surface measure on complex-valued $n \times n$ matrices $M$ with $M^*M = MM^*$ induced from the standard metric on $\C^{n\times n}$, and $\mathcal{Z}_n$ is a normalizing constant. In the presence of an external potential which grows sufficiently fast at infinity,
the eigenvalues accumulate on a compact set called the droplet as the size of matrix $n$ goes to infinity.
In the case of $Q(z)=|z|^2$, the system of eigenvalues is represented by the well-known complex Ginibre ensemble \cite{MR173726} and the droplet is the closed unit disk. With a proper scaling, the limit law of its spectral edge  follows the Gumbel distribution \cite{MR1986426}. Also, this result has been generalized to a class of radially symmetric potentials in \cite{MR3215627}.  

In this paper, we study random normal matrix ensembles with a hard edge. By localizing a potential to the droplet, we obtain a system of eigenvalues contained completely in the droplet. 
The growth of local droplets associated with a potential localized to subsets was studied in \cite{MR3056295} in connection with the two-dimensional Coulomb gas and Laplacian growth. 
The hard edge Ginibre ensemble, defined by localizing the potential $Q(z)=|z|^2$ to the unit disk so that all eigenvalues are constrained to range in the disk, was studied in \cite{MR3975882}. For more general potentials, edge behaviors of hard edge ensembles were investigated in \cite{MR4030288} for both regular and singular boundary points, e.g., cusps or double points. We also mention \cite{MR655364} for an earlier work on this type of hard edge ensembles.

We study limiting behaviors of the spectral radius for power potentials beyond the Ginibre case by finding a suitable rescaling factor and giving a proof based on the central limit theorem. We also prove the universality of this edge behavior for a class of radially symmetric potentials by applying the Laplace method from \cite{MR3215627}.

\subsection{The random normal matrix ensemble}
We consider an external potential $Q: \C\to \R\cup \{+\infty\}$ which is lower semi-continuous. We assume that $Q<+\infty$ on a set of positive area and $Q$ satisfies the growth condition
\begin{align}\label{Q:gr}
	\liminf_{z\to\infty} \frac{Q(z)}{\log |z|}>2.
\end{align}
In the random normal matrix model with potential $Q$, defined by the probability measure \eqref{ProbM}, the system of eigenvalues admits the joint probability density 
\begin{align}\label{eigen:pdf}
	P_{n}(z_1,\cdots,z_n)=\frac{1}{Z_n} \prod_{j\ne k}  |z_j -z_k| \,e^{-n \sum_{j=1}^{n} Q(z_j)}
\end{align}
with the normalization constant
\begin{align*}
Z_n = \int_{\C^n}  \prod_{j\ne k}  |z_j -z_k| \,e^{-n \sum_{j=1}^{n} Q(z_j)} \prod_{j=1}^n dA(z_j),
\end{align*}
where $dA(z)=dxdy/\pi$ is the normalized area measure on $\C$.
We refer to \cite{MR1643533, MR2172690, MR1986427} for earlier works on the random normal matrix model and its correlation structure.

It is well known that the eigenvalues $\{ z_j \}_{1}^{n}$ form a determinantal point process on $\C$, which means that for integers $k=1,\cdots,n$, its $k$-point correlation function 
\begin{align*}
\fR_{n,k}(z_1,\cdots,z_k) = \frac{n!}{(n-k)!}\int_{\C^{n-k}}P_n(z_1,\cdots, z_n) \, dA(z_{k+1})\cdots dA(z_{n})
\end{align*}
is expressed as a determinant 
\begin{align*}
\fR_{n,k}(z_1,\cdots, z_k) = \det \left[K_n(z_i,z_j)\right]_{i,j=1}^{k}
\end{align*}
with the kernel 
\begin{align*}
	K_n(z,w):= \sum_{k=0}^{n-1} p_k(z) \overline{p_k(w)} e^{-n (Q(z) + Q(w))/2},
\end{align*}
where $p_k$ is an orthonormal polynomial of degree $k$ with respect to the inner product 
\begin{align}\label{inner}
	\langle p,q\rangle_{nQ} := \int_{\C} p(z) \overline{q(z)} e^{-n Q(z)} dA(z).
\end{align}
This correlation structure described in terms of the reproducing kernel for the space of analytic polynomials with the inner product in \eqref{inner} plays an important role in the study of random normal matrix ensembles for global statistics including convergence of fluctuations of the spectral measure to gaussian fields \cite{MR2817648} and 
local statistics in the bulk and at the edge. In particular, on a local scale it is known that Ginibre kernel $$e^{z\bar{w}-|z|^2/2-|w|^2/2}$$ appears in the bulk of the spectrum. See  \cite[Section 7.5]{MR2817648} and \cite{MR3903320, MR1794066}. At a point on the boundary of the spectrum, the system of (suitably rescaled) eigenvalues converges to the determinantal point field with the correlation kernel 
 $$F(z+\bar{w})\,e^{z\bar{w}-|z|^2/2-|w|^2/2},$$ 
where $F$ is the free boundary plasma function $$F(z) = \frac{1}{2}\,\erfc\left(\frac{z}{\sqrt{2}}\right) = \frac{1}{\sqrt{2\pi}}\int_{-\infty}^{0}e^{-(z-t)^2/2}dt.$$ See \cite{MR3975882, 2017arXiv171006493H}.

\subsection{Droplets and potential theory}
With the growth assumption \eqref{Q:gr} on $Q$, the eigenvalues of random normal matrices are constrained to stay on a compact subset of $\C$. This confinement can be explained by the logarithmic potential theory. 


Following \cite{MR1485778}, we define the weighted logarithmic energy $I_{Q}(\mu)$ of a probability measure $\mu$ on $\C$ by
\begin{align*}
	I_{Q}(\mu)=\iint_{\C^2} \log \frac{1}{|z-\zeta|} d\mu(\zeta)d\mu(z) + \int_{\C} Q  \,d\mu.
\end{align*}
For a fixed potential $Q$, there is a unique probability measure $\sigma_Q$ which minimizes the weighted logarithmic energy among all compactly supported Borel probability measures on $\C$. The minimizer $\sigma_Q$ is called Frostman's equilibrium measure associated with $Q$. We set the logarithmic potential $U_Q(z)$ of $\sigma_Q$
\begin{displaymath}
	U_{Q}(z)= \int \log \frac{1}{|z-\zeta|} \, d\sigma_{Q}(\zeta)
\end{displaymath}
and the modified Robin constant
\begin{displaymath}
	F_Q= I_Q(\sigma_Q)-\frac{1}{2}\int Q\, d\sigma_{Q}.
\end{displaymath}
Then, the equilibrium measure has the following properties \cite{MR1485778}:
\begin{enumerate}
\item[(i)] $S= \mathrm{supp}(\sigma_Q)$ is compact and has positive capacity.
\item[(ii)]$2U_{Q}(z)+ Q(z) = 2F_Q $ holds for quasi-every $z$ in $S$.
\item[(iii)] If $Q$ is smooth in a neighborhood of $S$, then $d\sigma_Q= \1_{S} \, \pa \bar{\pa} Q \, dA$.
\end{enumerate}
Here, we write $\pa = \frac{1}{2}(\frac{\pa}{\pa x}-i\frac{\pa}{\pa y})$ for $z=x+i y$ and 
$\Delta = 4\pa \bar{\pa}= \frac{\pa^2}{\pa x^2}+ \frac{\pa^2}{\pa y^2}$. We call $S = S_Q$ the droplet in external potential $Q$. It is known that the eigenvalues of the random normal matrix ensemble tend to condensate to the droplet as $n$ goes to infinity. More precisely, the empirical distribution 
$\frac{1}{n}\sum_{j=1}^{n}\delta_{z_j}$ of the eigenvalues $z_j$ converges weakly to the equilibrium measure $\sigma_Q$. See \cite{MR3056295, MR1606719}.

\subsection[]{Main results}
We localize the external potential $Q$ to the droplet $S$ and write 
\begin{displaymath}
	Q^S (z)= 
	\begin{cases}
		Q(z),& \quad z\in S \\
		+\infty,& \quad z \in S^c.
	\end{cases}
\end{displaymath}
There still exists the equilibrium measure $\sigma_{Q^S}$ associated with $Q^S$ and it holds that $\sigma_{Q^S}= \sigma_{Q}$. For more details, see \cite[Section 5]{MR3056295}. 


We call the random matrix model with the localized potential $Q^S$ the \textit{hard edge ensemble} and the one without localization the \textit{free boundary ensemble}. In contrast to the free boundary case, eigenvalues of the hard edge ensemble tend to distribute completely inside the droplet. 

\begin{rmk*}
In the theory of random Hermitian matrices, the terminology ``hard edge'' is usually used for the case associated with the Bessel kernel; for example, in the Laguerre ensemble the hard edge appears at the origin. However, the ``hard edge'' in this note represents a slightly different situation, and a one-dimensional analogue can be found in \cite{MR2411911}, which is called ``soft/hard edge" explaining the situation that a soft edge meets a hard edge.
\end{rmk*}

\begin{figure} 
\includegraphics[width=.45\textwidth]{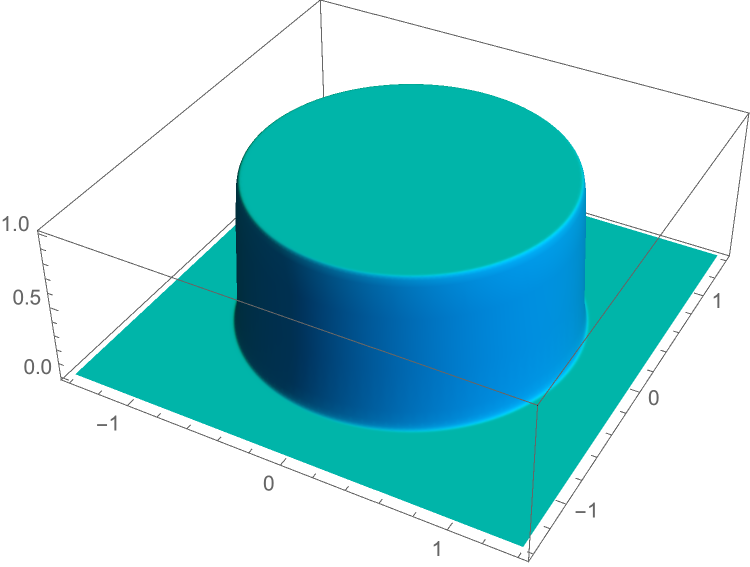}
\hspace{.05\textwidth}
\includegraphics[width=.45\textwidth]{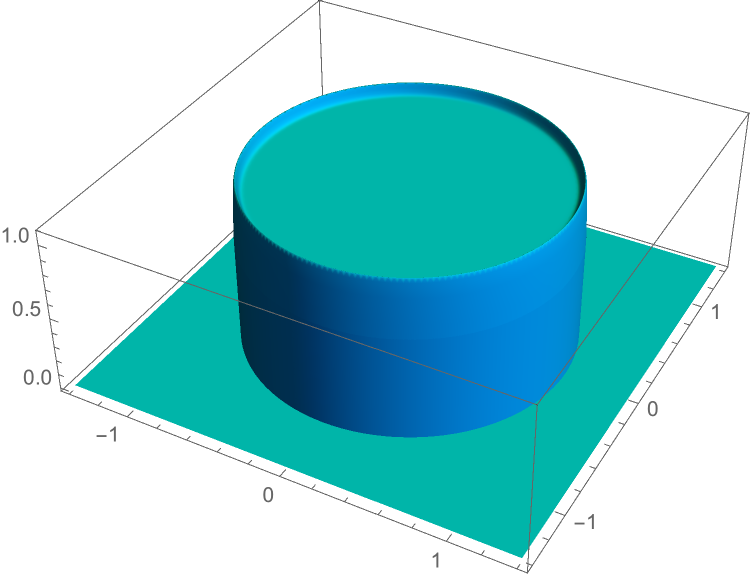}
\caption{Density profiles for the free boundary Ginibre ensemble (left) and the hard edge Ginibre ensemble (right) when $n=1000$.} \label{fig:R}
\end{figure}

For the special case when $Q(z)=|z|^2$, 
the equilibrium measure is the uniform measure on the unit disk. Figure \ref{fig:R} shows the graphs of the $1$-point density $\frac{1}{n}\fR_{n,1}$ for potentials $Q$ and $Q^S$, respectively. In both cases, the density $\frac{1}{n}\fR_{n,1}$ eventually converges to the equilibrium density $\1_{S}$ as $n$ goes to infinity.
However, on the local scale, the eigenvalues of the hard edge model are more densely distributed near the boundary of the droplet than those of the free boundary case.

We rescale the spectral radius at the boundary of the droplet in the inward direction. Let $\{ z_j \}_{1}^{n}$ be the eigenvalues for the hard edge ensemble associated with $Q^S$ and 
$|z|_n$ be the maximal modulus of the eigenvalues, i.e.,
$$|z|_n = \max_{1\leq j \leq n}|z_j|.$$ We first consider the case of the power potential $Q(z)=|z|^{2d} (d > 0)$. In this case, the droplet is the disk
$$S=\{ z\in \C\ : \  |z|\leq d^{-\frac{1}{2d}} \}$$
and the equilibrium measure is 
$$d\sigma(z) = d^2 |z|^{2d-2} \1_S(z)\,dA(z).$$
 
\begin{thm}\label{thm:power}
Let $Q(z) = |z|^{2d}$ with $d>0$. Let $\omega_n $ be the rescaled spectral radius of the hard edge ensemble associated with $Q^S$, which is defined by
\begin{displaymath}
	\omega_n = nd(d^{\frac{1}{2d}}|z|_n-1) \log 4.
\end{displaymath}
Then $\omega_n$ converges in distribution to the exponential distribution: 
\begin{displaymath}
	\lim_{n\to\infty}\P[\omega_n \leq \xi] = e^{\xi},\quad \xi\leq 0,
\end{displaymath}
uniformly for $\xi$ in every compact subset of $\R^{-}$.  
\end{thm}

We extend this result to general radially symmetric potentials. From now on, we assume that the potential is radially symmetric:
$Q(z)=q(|z|)$ 
where $q$ is a smooth function on $\R^+$. Also assume that $Q$ is subharmonic in $\C$ and strictly subharmonic in a neighborhood of the outer boundary of $S$. 
Due to the rotational symmetry, the boundary of the droplet $S$ consists of circles centered at the origin. We rescale the spectral radius $|z|_n$ about the radius of the outer boundary of $S$.

\begin{thm}\label{thm:radial} Let $R_0$ be the radius of the outer boundary of $S$ and $\delta = \pa \bar{\pa} Q(R_0)$.
Let $\omega_n $ be the rescaled spectral radius defined by 
\begin{displaymath}
	\omega_n = R_0 n \delta (|z|_n-R_0) \log 4.
\end{displaymath}
Then, the following convergence holds:
\begin{displaymath}
	\lim_{n\to\infty}\P[\omega_n \leq \xi] = e^{\xi},\quad \xi\leq 0,
\end{displaymath}
uniformly for $\xi$ in every compact subset of $\R^{-}$.  
\end{thm}

For finite $l$, we consider the distribution of the $l$-th largest modulus of eigenvalues. 
We write the limit law of the rescaled spectral radius in Theorem \ref{thm:radial} as follows:
\begin{align*}
	F_{\mathrm{hard}}(\xi):= e^{\xi} \quad \mbox{for} \quad \xi \leq 0.
\end{align*}
For the limit law $F_{\mathrm{hard}}$, define 
\begin{align*}
	& F^{(l)}_{\mathrm{hard}} {(\xi)}:= F_{\mathrm{hard}}(\xi) \sum_{k=0}^{l-1}\frac{1}{k!} 
	\left[ - \log {F_{\mathrm{hard}} (\xi)}\right]^{k}
	= e^{\xi} \sum_{k=0}^{l-1} \frac{(-\xi)^k}{k!}
	\quad \mbox{for} \quad \xi \leq 0.
\end{align*}

\begin{thm} \label{order:hard}
Let $\{z_k \}_1^{n}$ be the eigenvalues of the hard edge ensemble associated with $Q^S$ and $|z|_n^{(l)}$ be the $l$-th largest $|z_k|$. Let $R_0$ be the radius of the outer boundary of the droplet and $\delta=\pa\bar{\pa}Q(R_0)$.  Then,
for $\xi\leq 0$,
	\[ \lim_{n\to \infty}
	\P \left[ R_0 n \delta \Big(|z|^{(l)}_n -R_0\Big) \log 4 \leq \xi \right] 
	= F_{\mathrm{hard}}^{(l)}(\xi).
	\]
\end{thm}

\begin{rmk*}
For the hard edge Ginibre ensemble, it has been shown in \cite{MR3975882} that the system of rescaled eigenvalues $\tilde{z}_j = \sqrt{n}(z_j-1)$
converges to the determinantal point field on the left half plane $\L = \{z\in\C : \Re z \leq 0\}$ with the correlation kernel 
\begin{equation*}
H(z+\bar{w})\,\1_{\L}(z) \1_{\L}(w) \,e^{z\bar{w}-|z|^2/2-|w|^2/2}, 
\end{equation*}
where $H$ is the hard edge plasma function (see Figure \ref{fig:plasma})
$$H(z) := \frac{1}{\sqrt{2\pi}} \int_{-\infty}^{0} \frac{e^{-(z-t)^2/2}}{F(t)} dt \quad ; \quad F(t) = \frac{1}{\sqrt{2\pi}}\int_{t}^{\infty}e^{-{\xi}^2/2}d\xi=\frac{1}{2}\erfc\left(\frac{t}{\sqrt{2}}\right).$$
We also refer to \cite[Section 15]{MR2641363} and \cite{MR655364} for the calculation of the hard edge kernel.
 This result generalizes to the case of radially symmetric potentials in \cite{2018arXiv180806959A}. The normalizing constant $\log 2$ in Theorems 1.1 - 1.3 has a relation with the hard edge plasma function $H$. By a simple calculation, one finds that
$$H(0) = - \int_{-\infty}^{0} \frac{d}{dt} \log  \Big( \int_{t}^{\infty} e^{-\xi^2/2} d\xi \Big) dt = \log 2.$$
\end{rmk*}

\begin{figure}
\includegraphics[width=.45\textwidth]{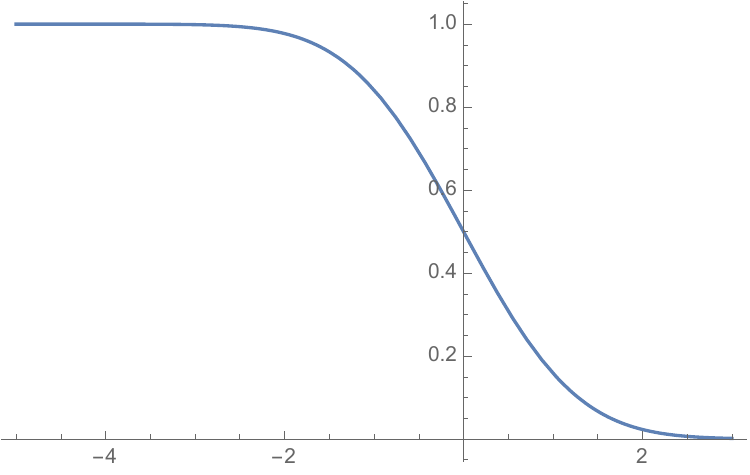}
\hspace{.05\textwidth}
\includegraphics[width=.45\textwidth]{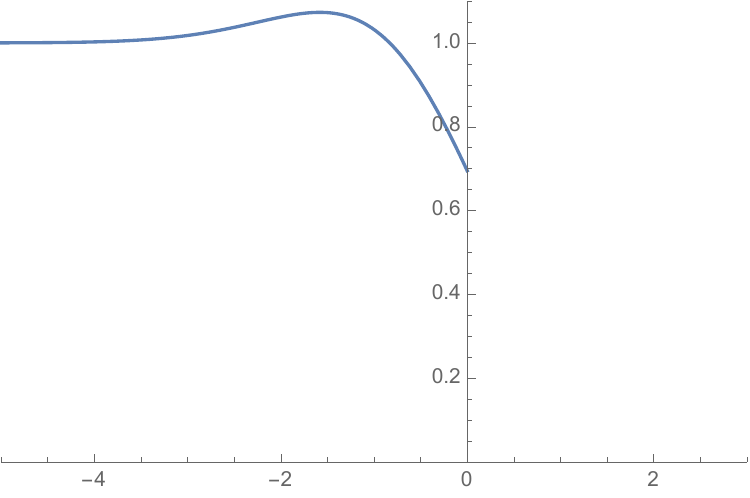}
\caption{The graphs of the free boundary plasma function $F$ (left) and the hard edge plasma function $H$ (right), restricted to $\R$.}  \label{fig:plasma}
\end{figure}

\subsection[]{Gap probabilities} \label{sec:gp}
In this subsection, we compute gap probabilities to obtain the distribution function of the spectral radius. The following computation is based on \cite{MR2514781, MR2129906}.

Let $\{z_j\}_1^n$ be the random normal matrix ensemble associated with the potential $Q$.
For $x\geq 0$, write $\Omega=\Omega_x =\{ z\in \C \ ; \ |z|>x \}$. Then the gap probability that none of the eigenvalues are contained in $\Omega$ is computed as 
\begin{align*}
	\P[|z|_n\leq x] &= \P[\mbox{no eigenvalues in }\Omega]\\
	&=\frac{1}{Z_n}\int \Big( \prod_{j=1}^{n}(1-\1_{\Omega}(z_j)) \Big) 
	e^{-n\sum Q(z_j)}\prod_{j\neq k}|z_j-z_k| \prod_{j=1}^{n} dA(z_j)\\
	&=\det \Big( \int_{\C} p_j(z) \overline{p_k(z)} (1-\1_{\Omega}(z)) 
	e^{-nQ(z)}dA(z)\Big)_{j,k=0}^{n-1},
\end{align*}
where $p_j$ is an orthonormal polynomial of degree $j$ with respect to $e^{-nQ} dA$. Each $p_j$ can be chosen to be a monomial since $Q$ is radially symmetric.  
Thus, the gap probability has a simplified form 
\begin{align}\label{gappr}
	\P[|z|_n\leq x]
	&=\det \Big( \delta_{j,k}-\int_{|z|>x } p_j(z)\overline{p_k(z)}e^{-nQ} dA \Big)_{j,k=0}^{n-1}\\
	&=\prod_{j=0}^{n-1}\Big(1-\int_{x}^{\infty} |p_j(r)|^2 e^{-nQ(r)} 2r dr \Big).\nonumber
\end{align}
In the hard edge case, we take $Q^S$ instead of $Q$. 

To find a limit of $\log \P_n[|z|_n\leq x]$ as $n\to\infty$,
we shall use the following lemma.

\begin{lem}\label{lem:parts}
Let $X_{n}$ be a subset of $\C$ such that $X_n = \{x_{n,1},\cdots, x_{n,n}\}$. If $ X_{n}$ satisfies the conditions
\begin{itemize}
\item[(a)] $\sum_{j=1}^{n} x_{n,j}=O(1)$
\item[(b)] $x_{n,j}=o(1)$  uniformly for all $1\leq j \leq n$
\end{itemize}
as $n\to\infty$, then we have
\begin{align*}
	\sum_{j=1}^{n} \log \left( 1-x_{n,j} \right)= -\sum_{j=1}^{n} x_{n,j}+o(1)
	\quad \mbox{as}\quad n \to \infty.
\end{align*}
\end{lem}
\begin{proof}
For sufficiently large $n$, there exists a function $f$ such that 
\begin{align*}
	\sum_{j=1}^{n} \log \left( 1-x_{n,j} \right)
	=-\sum_{j=1}^{n} \left(  x_{n,j} + x_{n,j}f(x_{n,j}) \right)
\end{align*}
and $f(x)=o(|x|)$ as $|x| \to 0$ by the Taylor series expansion. It is immediately obtained from the assumptions (a) and (b) that
\begin{align*}
	\sum_{j=1}^{n}  x_{n,j}f(x_{n,j}) = o(1)\quad \mbox{as}\quad n\to\infty,
\end{align*}
which proves the lemma.
\end{proof}

\subsection{Comments and related works}
\subsubsection{}
 In the free boundary case, the eigenvalues are admitted to be outside the droplet in contrast to the hard edge case. It is known that the spectral radius fluctuates slightly away from the droplet \cite{MR1986426, MR3215627}. For example, in the free boundary Ginibre case, the spectral radius $|z|_{n}$ is approximated by
$$|z|_{n} \simeq 1 + \sqrt{ \frac{\gamma_n}{4n}} + \frac{1}{\sqrt{4n\gamma_n}}\,G$$
where $\gamma_n = \log(n/2\pi) -2\log\log n$ and $G$ has a standard Gumbel distribution with the cumulative distribution function given by $\P[G\leq x]=e^{-{e^{-x}}}$ for $x\in \R$. On the other hand, the spectral radius of the hard edge ensemble fluctuates in a narrow range inside the droplet within a distance of order $n^{-1}$ from the boundary. 
 \subsubsection{}
When the external potential does not satisfy the growth condition \eqref{Q:gr}, the eigenvalues may not accumulate on a compact set. For example, the random normal matrix ensemble associated with the potential $Q(z)=\frac{n+1}{n}\log(1+|z|^2)$, also known as the spherical ensemble \cite{MR2552864}, the support of the equilibrium is the whole complex plane $\C$ and its density has a heavy tail. A class of potentials for which 
fluctuations of the spectral radius (modulus of the extremal particle) are heavy-tailed has been studied in \cite{2018arXiv181112225B, MR4079754}. We also mention that three universality classes, Gumbel, Fr\'{e}chet, and Weibull, for the statistics of the spectral radius have been investigated in \cite{MR3761607}.

\subsubsection{}
Random normal matrix ensembles are represented by two-dimensional Coulomb gases at a specific temperature where the system has a determinantal structure. 
Global properties of two-dimensional Coulomb gas at any temperature have been focused in several works, e.g., the global fluctuation of particles about the equilibrium 
has been studied in \cite{2016arXiv160908582B} and \cite{MR3788208}. The separation of Coulomb particles has been analyzed in \cite{MR3784540}, and a crossover behavior of them in a high temperature regime has been studied in \cite{MR3962973}. The distance between the extremal particle and the droplet has been investigated for general potentials in \cite{2019arXiv190700923A}. 

The hard edge setting described in present paper can be applied to a Coulomb gas in any dimension. In \cite{MR3664397}, the transition between ``pushed" and ``pulled" phases for Coulomb gases in any dimension has been studied, and the ''hard edge" in this paper is exactly the critical situation where the transition occurs. 
A natural question is to ask about the fluctuation of the maximal modulus of particles at the hard edge in any dimension.

\section[]{Proof of Theorem \ref{thm:power}}

In this section, we prove Theorem \ref{thm:power}. The proof is mainly based on the normal approximation to the gamma distribution. This approach is also found in the study of the free boundary ensembles \cite{MR3215627}. 

Consider the case when $Q(z)=|z|^{2d}$. The corresponding droplet is the disk
$$S=\{z\in\C : |z| \leq d^{-1/2d}\},$$
and the localized potential $Q^S$ is defined by 
\begin{align*}
	Q^S (z)= 
	\begin{cases}
		|z|^{2d},& \quad |z|\leq d^{-\frac{1}{2d}} \\
		+\infty,& \quad |z |> d^{-\frac{1}{2d}}.
	\end{cases}
\end{align*}
The orthonormal polynomial of degree $j$ with respect to the measure $e^{-nQ^S}dA$ is 
\begin{displaymath}
	p_j(z)= \left(\frac{d\,n^{\frac{j+1}{d}}}{\gamma(\frac{j+1}{d},\frac{n}{d})}\right)^{\frac{1}{2}} z^j,
\end{displaymath}
where $\gamma(s,\alpha)= \int_{0}^{\alpha} t^{s-1}e^{-t}dt $ is the lower incomplete gamma function. Now we define the rescaled spectral radius by 
\[
	\omega_n = dn\left(d^{\frac{1}{2d}}|z|_n -1\right)\log4
\]
and the cumulative distribution function by
$F_n (\xi)= \P[\omega_n \leq \xi]$ for $\xi \leq 0$.
Write $x=d^{-\frac{1}{2d}}\left(1+\frac{\xi}{dn\log 4}\right)$. By \eqref{gappr},
we calculate $\log F_n(\xi)$ as follows:
\begin{align}\label{logsum}
	\log F_n (\xi)
	&=\log \P[|z|_n \leq x]\\ \nonumber
	&=\log \prod_{j=0}^{n-1} \Big( 1 -
	\int_{x<|z|\leq  d^{-\frac{1}{2d}}} 
	|p_j(z)|^2 e^{-n|z|^{2d}} dA(z)\Big)\\ \nonumber
	&=\log \prod_{j=0}^{n-1} \Big(1- \int_{nx^{2d}}^{\frac{n}{d}} 
	\frac{t^{\frac{j+1}{d}-1}}{\gamma \big(\frac{j+1}{d},\frac{n}{d}\big)} e^{-t} dt \Big)\\ \nonumber
	&= \sum_{j=0}^{n-1}\log\Big(1- \int_{\frac{\xi}{d\log2}+o(1)}^{0} 
	\Big( t+ \frac{n}{d} \Big) ^{\frac{j+1}{d}-1} 
	\frac{e^{-\left(t+\frac{n}{d}\right)}}{\gamma\big(\frac{j+1}{d},\frac{n}{d}\big)} dt \Big),\nonumber
\end{align}
where $o(1)\to 0$ uniformly for $\xi$ in every compact subset of $\R^{-}$ as $n\to \infty$. 

Let $\Phi$ and $\phi$ denote the cumulative distribution function and the probability density function of the standard gaussian distribution respectively, i.e.,
\begin{equation}\label{cdfpdf}
\Phi(x) = \frac{1}{\sqrt{2\pi}}\int_{-\infty}^{x} e^{-t^2/2} dt\quad;\quad \phi(x)=\frac{1}{\sqrt{2\pi}}e^{-x^2/2}.
\end{equation}

Throughout the section, let $n_0$ denote $\sqrt{n}\log n$.
\begin{lem}\label{lem:app}
For $0\leq k \leq n_0$, we have the following approximations:
\begin{align*}
	\frac{\gamma\left(\frac{n-k}{d},\frac{n}{d}\right)}{\Gamma\left(\frac{n-k}{d} \right)}
	&= \Phi\left(\frac{k}{\sqrt{(n-k)d}}\right)+O\left(n^{-1/2}\right)
	, \\
	\left(t+\frac{n}{d}\right)^{\frac{n-k}{d}-1} \frac{e^{-\left(t+\frac{n}{d}\right)}}
	{\Gamma\left(\frac{n-k}{d}\right)}
	&= \sqrt{\frac{d}{n-k}}\phi\left(\frac{k}{\sqrt{(n-k)d}}\right) \left(1+O\left(n^{-1/2}\log^3 n\right)\right) 
\end{align*}
as $n\to \infty$, where the error terms are uniform in $k$ with $0\leq k \leq n_0$ and $t$ on every compact subset of $\R^{-}$. Furthermore, $\Gamma\left(\frac{n-k}{d} \right) /
\gamma\left(\frac{n-k}{d},\frac{n}{d}\right) $ is uniformly bounded for all $k$ with $0\leq k \leq n-1$ and sufficiently large $n$.
\end{lem}

\begin{proof}

Let $U_{m}$ be a random variable which follows a Gamma distribution 
$\Gamma\left(\frac{m}{d},1 \right)$. If we write $G_{m}$ and $g_{m}$ for the cumulative distribution function and the probability density function of $U_{m}$ respectively, then we have
\begin{align*}
G_{n-k} \left(\frac{n}{d}\right)=\frac{\gamma\left(\frac{n-k}{d},\frac{n}{d}\right)}{\Gamma\left(\frac{n-k}{d} \right)}\quad ; \quad
g_{n-k}\left(t+\frac{n}{d} \right)=\left(t+\frac{n}{d}\right)^{\frac{n-k}{d}-1} \frac{e^{-\left(t+\frac{n}{d}\right)}}{\Gamma\left(\frac{n-k}{d}\right)}.
\end{align*}
Fix an integer $k$ with $0\leq k \leq \sqrt{n}\log n$. Since $Y= \sqrt{\frac{d}{n-k}}\left( U_{n-k}-\frac{n-k}{d} \right)$ is asymptotically normal, we have the following approximation by the Berry-Esseen theorem:
\begin{align}\label{eqn:cdf}
	G_{n-k}\left( \frac{n}{d} \right)
	=\P\left[{\sqrt{\frac{n-k}{d}}}\ Y < \frac{k}{d} \right]
	= \Phi\left(\frac{k}{\sqrt{(n-k)d}}\right)
	+O \left( n^{-1/2} \right)
\end{align}
as $n\to \infty$. Here, note that the $O$-constant can be taken to be uniform in $k$ with $0\leq k \leq \sqrt{n}\log n$. Now write $f_{Y}$ for the probability density function of the random variable $Y$. By the Edgeworth expansion for $f_{Y}$, we obtain 
\begin{align*}
f_{Y}(x) = \phi(x)\left(1+ \frac{c}{\sqrt{n}}H_3(x)+O\left(n^{-1}\right)\right),
\end{align*}
where $H_3$ is the Hermite polynomial of order $3$ and $c$ is a constant. 
 This gives 
\begin{align}\label{eqn:pdf}
	&g_{n-k} \left( t+ \frac{n}{d} \right)\\
	&=\sqrt{\frac{d}{n-k}} \,\phi\left(\frac{k+td}{\sqrt{(n-k)d}}\right)
	\left( 1+ \frac{c}{\sqrt{n}}H_3\left( \frac{k+td}{\sqrt{(n-k)d}} \right)  + O\left(n^{-1}\right) \right)
 \nonumber,
\end{align}
which proves the first statement of the lemma. 

To show the uniform boundedness of $\Gamma\left(\frac{n-k}{d} \right) /\gamma\left(\frac{n-k}{d},\frac{n}{d}\right)$, we observe that
for $1\leq j\leq n-1$
\[
	\P\left[U_j <\frac{n}{d}\right]\geq \P\left[U_{j+1}<\frac{n}{d}\right].
\] 
By (\ref{eqn:cdf}), we have
\[
	\lim_{n\to\infty}\P\left[U_{n}<\frac{n}{d}\right] = \P[Z<0]=\frac{1}{2},
\] 
where $Z$ is the standard normal distribution. This gives the following inequality: for all $k$ with $0\leq k\leq n-1$,
\[
\frac{\Gamma\left(\frac{n-k}{d}\right)}{\gamma\left(\frac{n-k}{d},\frac{n}{d}\right)}
\leq 2+o(1),
\]
where $o(1)\to 0$ uniformly in $k$ as $n\to \infty$ .
\end{proof}

Returning to \eqref{logsum}, we consider the following sums:
\begin{align*}
S_n :=  \sum_{0 \leq k \leq n_0} \! \left(t+\frac{n}
	{d}\right)^{\frac{n-k}{d}-1} \!
	\frac{e^{-\left(t+\frac{n}{d}\right)}}{\gamma\left(\frac{n-k}{d},\frac{n}{d}\right)},\quad
\epsilon_n:= \! \sum_{n_0< k < n} \!\left(t+\frac{n}
	{d}\right)^{\frac{n-k}{d}-1} \!
	\frac{e^{-\left(t+\frac{n}{d}\right)}}{\gamma\left(\frac{n-k}{d},\frac{n}{d}\right)}.
\end{align*}

\begin{lem}\label{lem:sum}
As $n\to \infty$, we have
\begin{align*}
	S_n \to d\log 2 \quad \mbox{and} \quad
	\epsilon_{n} \to 0
\end{align*}
locally uniformly for $t\in \R^{-}$.
\end{lem}

\begin{proof}
By Lemma \ref{lem:app}, we have
\begin{align*}
	S_n &= \sum_{0\leq k \leq n_0} \left(\frac{\Gamma\left(\frac{n-k}{d} \right)}
	{\gamma\left(\frac{n-k}{d},\frac{n}{d}\right)} \left(t+\frac{n}{d}\right)^{\frac{n-k}{d}-1}
	 \frac{e^{-\left(t+\frac{n}{d}\right)}}{\Gamma\left(\frac{n-k}{d}\right)}\right)\\
	 &= \sum_{0\leq k\leq n_0} \sqrt{\frac{d}{n}}\,\left(\Phi\left(\frac{k}{\sqrt{nd}}\right)\right)^{-1}\phi\left(\frac{k}{\sqrt{nd}}\right)\left(1+o(1)\right)
\end{align*}
as $n\to \infty$. By the Riemann sum approximation, we obtain
\begin{align*}
	\lim_{n\to\infty}S_n = \frac{d}{\sqrt{2\pi}} \int_0^{\infty} \frac{e^{-\frac{x^2}{2}}}{\Phi(x)} dx.
\end{align*}
Integration by parts gives  
\begin{align*}
	\frac{d}{\sqrt{2\pi}} \int_0^{\infty} \frac{e^{-\frac{x^2}{2}}}{\Phi(x)} dx 
	= d \int_{0}^{\infty} \frac{d}{ds} \log \Phi(s) ds 
	= d \log 2,
\end{align*}
which proves the convergence of $S_n$ to $d\log 2$.

Now fix an integer $k$ with $n_0 < k \leq n-1$. We observe that there exists a positive constant $C$ such that
\begin{align}\label{ginq}
g_{n-k}\left(t+\frac{n}{d}\right) \leq C \, g_{n-n_0}\left(t+\frac{n}{d}\right).
\end{align}
Indeed, the inequality
\begin{align*}
	\frac {g_{n-n_0}\left(t+\frac{n}{d}\right)}{g_{n-k}\left(t+\frac{n}{d}\right) }& 
	= \left(t+\frac{n}{d} \right)^{\frac{k}{d}-\frac{n_0}{d}} \frac{\Gamma\left(\frac{n-k}{d}\right) }{\Gamma\left(\frac{n-n_0}{d}\right)}\\
	&\geq \int_{0}^{t+\frac{n}{d}} 
	\frac{x^{\frac{n-n_0}{d}-1}e^{-x}}{\Gamma\left( \frac{n-n_0}{d} \right)}dx
	=\P \left[ U_{{n-n_0}}<t+\frac{n}{d}\right]
\end{align*} 
implies that for sufficiently large $n$,
\begin{align*}
	g_{n-n_0}\left( t+\frac{n}{d} \right) 
	\geq \left(\frac{1}{2}+o(1) \right) g_{n-k} \left( t+\frac{n}{d} \right)
\end{align*} 
where $o(1)$ is uniform for $k$ with $n_0 < k \leq n-1$ and for $t$ in every compact subset of $\R^{-}$.
By \eqref{eqn:pdf}, we have
\begin{align*}
	g_{{n-n_0}}\left( t+\frac{n}{d} \right) \leq 
	c_1 n^{-1/2} e^{-c_2(\log n)^2} 
\end{align*}
for some positive constants $c_1$ and $c_2$. By \eqref{ginq} and the uniform boundedness of $(G_{n-k}(n/d))^{-1}$ shown in Lemma \ref{lem:app}, there exists a constant $C'$ such that
$$\epsilon_n = \sum_{n_0<k<n}\frac{g_{n-k}\left(t+\frac{n}{d}\right)}{G_{n-k}\left(\frac{n}{d}\right)} \leq C' n^{1/2} e^{-c_2(\log n)^2},$$
which proves the convergence of $\epsilon_n$ to $0$.
\end{proof}

\begin{proof}[Proof of Theorem \ref{thm:power}]
Returning to the sum in \eqref{logsum}, we observe from Lemma \ref{lem:sum} that 
\begin{align}\label{conA}
	\sum_{j=0}^{n-1} 
	\int_{\frac{\xi}{d\log2}}^{0} \!\! \left( t + \frac{n}{d}\right)^{\frac{j+1}{d}-1}
	\!\! \frac{e^{-t-\frac{n}{d}}}{\gamma\Big(\frac{j+1}{d},\frac{n}{d}\Big)}dt 
	= \int_{\frac{\xi}{d\log2}}^{0} S_n + \epsilon_{n}\, dt \to -\xi
\end{align}
as $n\to \infty$. We also obtain that 
\begin{align}\label{conB}
	\lim_{n\to\infty}\int_{\frac{\xi}{d\log2}}^{0} \left( t+ \frac{n}{d} \right)^{\frac{j+1}{d}-1} 
	\frac{e^{-t-\frac{n}{d}}}{\gamma\Big(\frac{j+1}{d},\frac{n}{d}\Big)} dt = 0
\end{align} 
uniformly in $j$ with $0\leq j\leq n-1$ 
by Lemmas \ref{lem:app} and \ref{lem:sum}. By Lemma \ref{lem:parts}, it follows from \eqref{conA} and \eqref{conB} that 
\begin{align*}
	\lim_{n\to\infty}\log F_n(\xi)= \xi .
\end{align*}
uniformly on compact subsets of $\R^{-}$.
\end{proof}


\section[]{Universality for the hard edge ensembles}

In this section, we prove Theorem \ref{thm:radial}.
We consider a radially symmetric potential $Q(z)= q(|z|)$ with the following assumptions:
\begin{enumerate}
\item[(i)] $q : \R^{+}\to \R$ is smooth;
\item[(ii)] $Q$ is subharmonic on $\C$;
\item[(iii)] $Q$ is strictly subharmonic on a neighborhood of the outer boundary of $S$.
\end{enumerate}

Since $\Delta Q(r) \geq 0$ on $(0,\infty)$, $rq'(r)$ is increasing on $(0,\infty)$.
Following Saff and Totik \cite[Section IV.6]{MR1485778}, 
let $r_0$ be the smallest number that $q'(r)>0$ for all $r>r_0$ and $R_0$ be the smallest solution to $R_0 \,q'(R_0)=2$. Then the droplet $S$ is the ring
\begin{align*}	
	S= \{z\in\C\ ;\  r_0 \leq |z| \leq R_0 \}.
\end{align*}
Also note that $\{ rq'(r) \; | \; r_0 \leq r \leq R_0 \}=[0,2]$, which implies that for each $\alpha \in [0,2]$, there exists $r_{\alpha} \in [r_0, R_0]$ such that $r_\alpha q'(r_\alpha)=\alpha$. 

Let $\omega_n$ be the rescaled spectral radius defined by
$$ \omega_n = n C_0 (|z|_n - R_0)$$ 
where $C_0 = R_0 \delta \log 4$ and $\delta=\pa\bar{\pa}Q(R_0)>0$. Write $F_n$ for the cumulative distribution function of $\omega_n$ i.e., $F_n(\xi)=\P[\omega_n\leq \xi]$ for $\xi \leq 0$. 

Now we consider the functions $V$ and $V_k$ where 
\begin{align*}
	V(r)&=q(r)- 2\log r\\
	V_k(r)&=q(r) - \left(2- \frac{2k+1}{n} \right)\log r.
\end{align*}
Setting $k=(n-1)-j$, we have
\begin{align*}
	F_n(\xi)&=\prod_{j=0}^{n-1} 
	\left(1- \int_{x}^{R_0} |p_j (r)|^2 e^{-nq(r)} 2r dr \right)
	=\prod_{k=0}^{n-1} 
	\left( 1-\frac{\int_{x}^{R_0} e^{-nV_k (r)} dr }{\int_{r_0}^{R_0} e^{-nV_k (r)} dr} \right),
\end{align*}
where $x=R_0+ (nC_0)^{-1}\xi$. We write 
\begin{align}\label{x}
	x_{n,k}:=\frac{\int_{x}^{R_0} e^{-nV_k (r)} dr }{\int_{r_0}^{R_0} e^{-nV_k (r)} dr}
\end{align}
and show that $\{ x_{n,k} \}$ satisfies the conditions in Lemma \ref{lem:parts}. As in the previous section, we use the notation $n_0=\sqrt{n}\log n$ throughout this section.

\begin{lem} \label{lem:ms}
We have
\begin{equation}\label{mainsum}
	\sum_{0\leq k \leq n_0} x_{n,k} = - \xi + o(1),
\end{equation}
where $o(1) \to 0$ uniformly for $\xi$ in every compact subset of $\R^{-}$ as $n\to \infty$.
\end{lem}

\begin{proof}
We observe 
\begin{align*}
	V_k '(r)&= \frac{1}{r} \left( rq'(r) - \left(2-\frac{2k+1}{n} \right) \right), \\
	V_k ''(r)&= \Delta Q(r)  - \frac{1}{r^2} \left( rq'(r) - \left(2-\frac{2k+1}{n} \right) \right).
\end{align*}
We set $f(r)= rq'(r)$. Then, $f$ is increasing in $(0,\infty)$ and $f'>0$ in a neighborhood of $R_0$.
For each $k$ with $0\leq k \leq n_0$, there exists a unique $t_k \in [r_0, R_0]$ such that 
\begin{equation}\label{def:tk}
	f(t_k)=2-\frac{2k+1}{n},
\end{equation}
and $t_k$ is decreasing in $k$. Note that $f'(R_0) = R_0 \Delta Q(R_0) = 4\delta R_0$.
From the Taylor series expansion 
\begin{equation*}
	f(r) - f(R_0) = 4\delta R_0(r - R_0) + O(|r - R_0|^2),\quad r \to R_0,
\end{equation*}
we obtain 
\begin{align}\label{dist}
	R_0 - t_k= \frac{1}{4\delta R_0 }\frac{2k+1}{n}+O\left( \frac{k}{n} \right)^2.
\end{align}
Again, Taylor's theorem gives that there exists $t^*_{k,r}\in (r,t_k)$ such that 
\begin{align}\label{taylor:Vk}
	V_k(r)
	&=V_k(t_k)+\frac{1}{2}\Delta Q(t_k)(r-t_k)^2+ \frac{1}{6} V'''_k(t^*_{k,r}) |r-t_k|^3.
\end{align}
Here we note that 
\begin{align} \label{deltaQ}
	\Delta Q(t_k)&=\Delta Q(R_0)+ O\left( n^{-1/2} \log n \right),
\end{align}
where the error term is uniform for $k$ with $0\leq k\leq n_0$. Since 
\begin{align*}
	V_{k}'''(r)= V'''(r) + \left( \frac{2k+1}{n} \right) \frac{2}{r^3},
\end{align*}
we can choose sufficiently small $\epsilon>0$ such that there exists a constant $M$ satisfying
\begin{align*}
	\sup_{t\in[R_0-\epsilon, R_0]} |V_{k}'''(t)| < M.
\end{align*}
for all $k$ with $0\leq k\leq n_0$. Write $\epsilon_n = n^{-1/2}\log n$ and $$\zeta_k = \frac{2k+1}{2R_0 \sqrt{n\delta}}.$$
By \eqref{dist}, \eqref{taylor:Vk} and \eqref{deltaQ}, we obtain that 
for $k$ with $0\leq k\leq n_0$
\begin{align}\label{asym:1}
	\int_{t_k-\epsilon_n}^{R_0} e^{-n V_k(r)  }dr 
	&= e^{-nV_k(t_k)}\int_{-\infty}^{\zeta_k} e^{-s^2/2} \frac{ds}{\sqrt{4n\delta }}\, 
	(1+o(1)),\end{align}
where $o(1)\to 0$ uniformly in $k$ as $n\to \infty$. On the other hand, since $V'_k(t_k)=0$ and $V'_k(r)\leq 0$ for $r\leq t_k$,
there exist positive constants $c$ and $C$ such that 
\begin{align}\label{asym:2}
	\int_{r_0}^{t_k-\epsilon_n} e^{-n\left( V_k(r) - V_k(t_k) \right)}dr 
	&\leq C e^{-n \left( V_k\left(t_k-\epsilon_n \right) -V_k(t_k) \right)}
	\leq C e^{-c(\log n)^2}.
\end{align}
Combining \eqref{asym:1} and \eqref{asym:2}, we obtain the following asymptotics:
\begin{align}\label{asym:dm}
	\int_{r_0}^{R_0} e^{-nV_k(r)} dr = \sqrt{\frac{\pi}{2n\delta}} \, e^{-nV_k(t_k)} \,\Phi(\zeta_k)\, (1+o(1)),
\end{align}
where $\Phi$ is the cumulative distribution function of the standard gaussian distribution defined in \eqref{cdfpdf} and $o(1)\to 0$ uniformly in $k$ as $n\to \infty$.

Now we observe that by \eqref{dist}, \eqref{taylor:Vk}, and \eqref{deltaQ},
\begin{align}\label{asym:nm}
	\int_{R_0+\frac{\xi}{nC_0}}^{R_0} e^{-n V_k(r)} dr
	&= e^{-nV_k(t_k)} \int_{\zeta_k+ \frac{\sqrt{4\delta}}{\sqrt{n}C_0}\xi }^{\zeta_k} 
	e^{-s^2/2} \frac{ds}{\sqrt{4n\delta}} \,\left( 1+o(1) \right)\\ 
	&= - (nC_0)^{-1}\xi \cdot  e^{-nV_k(t_k)-\zeta_k^2/2} 
	\left( 1+o(1) \right),\nonumber
\end{align}
where $o(1)\to 0$ uniformly for $k$ with $0\leq k\leq n_0$ and for $\xi$ in every compact subset of $\R^{-}$ as $n\to \infty$. Combining the asymptotics \eqref{asym:dm} and \eqref{asym:nm} gives 
\begin{align*}
	\sum_{0\leq k \leq n_0}  x_{n,k} = - \frac{\sqrt{2\delta}}{\sqrt{n\pi}\,C_0}\xi \,(1+o(1)) \sum_{k=0}^{n_0} 
	\frac{e^{-\zeta_k^2/2}}{\Phi(\zeta_k)}.
\end{align*}
By the Riemann sum approximation with step length $(R_0\sqrt{n\delta})^{-1}$, we obtain
\begin{align*}
	\frac{1}{R_0\sqrt{n\delta}}\sum_{0\leq k \leq n_0} 
	\frac{e^{-\zeta_k^2/2}}{\Phi(\zeta_k)} = \int_{0}^{\infty} \frac{e^{-t^2/2}}{\Phi(t)} \, dt + o(1) = \sqrt{2\pi} \log 2 +o(1)
\end{align*}
as $n\to \infty$. Recall that the constant $C_0$ is defined by $$C_0 = R_0 \delta \log 4,$$
and we conclude that \eqref{mainsum} holds.
\end{proof}

The sum of $x_{n,k}$ over $k$ from $n_0$ to $n-1$ vanishes as $n\to \infty$. The following lemma gives the desired estimate.

\begin{lem}\label{lem:es}
We have
\begin{align*}
\sum_{n_0<k\leq n-1} x_{n,k} = o(1),
\end{align*}
where $o(1)\to 0$ locally uniformly for $\xi \in \R^{-}$ as $n\to\infty$.
\end{lem}

\begin{proof}
Write $n_1$ for the number $\frac{1}{2}\sqrt{n} \log n$ and $t_{n_1}$ for the unique solution to $$f(t)=2-\frac{2n_1+1}{n}$$
as defined in \eqref{def:tk}. Write $\epsilon_n$ for the number $n^{-1/2}\log n$.
It follows from \eqref{dist} that 
\begin{equation*}
	t_{n_1} - t_{n_0} \geq c_0 \,\epsilon_n
\end{equation*}
for some positive constant $c_0$. Using this, we obtain that for $n_0 < k \leq n-1$
\begin{align}\label{ine:rR}
	\int_{r_0}^{R_0} e^{-nV_k(r)} dr \geq \int_{t_{n_0}}^{t_{n_1}} e^{-nV_k(r)} dr \geq c_0 \, \epsilon_n \,e^{-nV_k(t_{n_1})},
\end{align}
where the constant $c_0$ is uniform in $k$.  

On the other hand, we observe that for $k$ with $n_0 < k \leq n-1$
\begin{align}\label{vkvn}
	V_{k}(r) - V_k(t_{n_1}) = V_{n_1}(r) - V_{n_1}(t_{n_1}) - \frac{2(k-n_1)}{n}\log \frac{t_{n_1}}{r}.
\end{align}
Since $R_0 - t_{n_1} = c_1\, \epsilon_n + O(\epsilon_n)^2$ for some positive constant $c_1$, there exists a constant $c_2$ such that
for $s$ in a compact subset in $\R^{-}$   
\begin{align*}
	V_{k}\left(R_0 + \frac{s}{n} \,\right) - V_k(t_{n_1}) \geq V_{n_1}\left(R_0 + \frac{s}{n}\,\right) - V_{n_1}(t_{n_1}) \geq c_2 \, \epsilon_n^2
\end{align*}
by \eqref{vkvn}. It follows from \eqref{ine:rR} that for all $k$ with $n_0 < k\leq n-1$ and $x = R_0 + (nC_0)^{-1}\xi$ 
\begin{align*}
	x_{n,k} \leq (c_0\,\epsilon_n)^{-1}\int_{x}^{R_0} e^{-n(V_k(r)- V_{k}(t_{n_1}))} \,dr \leq C n^{-1/2} e^{-c (\log n)^2}
\end{align*}
for some constants $c$, $C$ which is uniform for $\xi$ in every compact subset in $\R^{-}$. Thus we have
\begin{align*}
	\sum_{n_0\leq k \leq n-1} x_{n,k} =o(1),\quad n\to \infty,
\end{align*}
which completes the proof.
\end{proof}

\begin{proof}[Proof of Theorem \ref{thm:radial}]
Combining Lemma \ref{lem:ms} and Lemma \ref{lem:es}, we obtain that 
\begin{align*}
	\sum_{k = 0}^{n-1} x_{n,k} = - \xi + o(1),
\end{align*}
where $o(1)\to 0$ locally uniformly for $\xi \in \R^{-}$ as $n\to \infty$. It also follows from Lemma \ref{lem:es} that $x_{n,k}\to 0$ uniformly for all $k$ with $n_0\leq k \leq n-1$ as $n\to \infty$. For all $k$ with $0\leq k \leq n_0$, we have a uniform error bound $x_{n,k} = O(n^{-1/2})$ by \eqref{asym:dm} and \eqref{asym:nm}. Hence, by Lemma \ref{lem:parts}, we conclude that 
\begin{align*}
F_n(\xi) = -\xi + o(1),
\end{align*}
where $o(1)\to 0$ locally uniformly in $\R^{-}$. 
\end{proof}




\section[]{Order statistics}

In this section, we examine the limit law of the $l$-th modulus of eigenvalues of random normal ensembles with a hard edge. Our approach is inspired by the earlier work in \cite{MR2035641}, where the order statistics for the free boundary Ginibre ensemble was studied.

\begin{proof}[Proof of Theorem \ref{order:hard}]
Let $\{z_j\}_1^n$ be the random normal matrix ensemble associated with the potential $Q^S$. 
We consider the probability $p_{n,k}(x)$ that exactly $k$ eigenvalues are in the region $\{z\in\C : |z|>x\}$. It is computed as follows: 
\begin{align*}
	p_{n,k}(x)
	=\frac{1}{k!} \left( \frac{d}{d\lambda}\right)^{k} \bigg| _{\lambda=-1} 
	\E \prod _{j=1}^{n} \left(1+\lambda \mathbf{1}_{\{|z|>x\}}(z_j) \right).
\end{align*}
See \cite{MR2514781}. It follows from the computation in Section \ref{sec:gp} that 
\begin{align*}
	\E \prod _{j=1}^{n} \left(1+\lambda \mathbf{1} _{\{|z|>x\}}(z_j) \right) 
	&=\prod_{j=0}^{n-1} \left(1+ \lambda \int_{\{|z|>x\}} |p_j(z)|^2 
	e^{-n Q^S (z)} dA(z) \right),
\end{align*}
where $p_j$ is an orthonormal polynomial of degree $j$ with respect to $e^{-nQ^S}dA$. For $\xi \leq 0$ and $x=  R_0+ \xi/(nC_0)$ with $C_0=R_0\delta\log4$, we write
\begin{align*}
	x_{n,j}= \int_{\{|z|>x\}} |p_j(z)|^2  e^{-n Q^S (z)} dA(z)
\end{align*}
and $g_n(\lambda)=\prod_{j=0}^{n-1}\left(1+\lambda x_{n,j}\right)$ as a function of $\lambda$ in $\C$. We observe that the distribution function of the $l$-th largest modulus $|z|_{n}^{(l)}$ can be written as follows:
\begin{align*}
	\P \left[|z|_{n}^{(l)}\leq x\right]
	&= \sum_{k=0}^{l-1}\  p_{n,k}(x) = \sum_{k=0}^{l-1}\frac{1}{k!} \left( \frac{d}{d\lambda}\right)^{k} \bigg| _{\lambda=-1} g_n(\lambda).
\end{align*}

It is sufficient to show that 
	\begin{align}\label{eqn:order}
	\lim_{n\to\infty} \left( \frac{d}{d\lambda}\right)^{k}\bigg|_{\lambda=-1} g_n(\lambda) 
	= F_{\mathrm{hard}}(\xi)\left[- \log F_{\mathrm{hard}}(\xi) \right]^k,
	\end{align}
where $F_{\mathrm{hard}}(\xi) = e^{\xi}$ as defined before Lemma \ref{order:hard}. 
In Section 3, we have shown that $\{ x_{n,j} \}$ satisfies the conditions (a), (b) in Lemma \ref{lem:parts} and $\sum_{j=0}^{n-1} x_{n,j}$ converges to $-\log F_{\mathrm{hard}}(\xi)$. By applying Lemma \ref{lem:parts} to $\{ \lambda x_{n,j} \}$, 
we obtain 
\begin{align*}
	g_n(\lambda) \ \to e^{-\lambda \log  F_{\mathrm{hard}}(\xi) }
\end{align*}
uniformly on every compact subset of $\C$ as $n\to \infty$. Since $g_n$ is an analytic function, Cauchy integral formula implies 
\begin{align*}
	g_n^{(k)}(\lambda) \to \left[-\log F_{\mathrm{hard}} (\xi) \right]^k 
	e^{-\lambda \log F_{\mathrm{hard}}(\xi) }
\end{align*}
uniformly on every compact subset of $\C$. Taking $\lambda = -1$, we prove (\ref{eqn:order}).
\end{proof}

\subsection*{Acknowledgements}

The author thanks Nam-Gyu Kang for helpful advice and discussions. This work was partially supported by the KIAS Individual Grant (MG063103) at Korea Institute for Advanced Study and by the National Research Foundation of Korea Grant funded by the Korea government (2019R1F1A1058006).

\bibliographystyle{habbrv}
\bibliography{bib}

\begin{thebibliography}{10}

\bibitem{MR3962973}
G.~Akemann and S.-S. Byun.
\newblock The high temperature crossover for general 2{D} {C}oulomb gases.
\newblock {\em J. Stat. Phys.}, 175(6):1043--1065, 2019.

\bibitem{2019arXiv190700923A}
Y.~{Ameur}.
\newblock {A localization theorem for the planar Coulomb gas in an external
  field}.
\newblock arXiv:1907.00923.

\bibitem{2018arXiv180806959A}
Y.~{Ameur}.
\newblock {A note on normal matrix ensembles at the hard edge}.
\newblock arXiv:1808.06959.

\bibitem{MR3784540}
Y.~Ameur.
\newblock Repulsion in low temperature {$\beta$}-ensembles.
\newblock {\em Commun. Math. Phys.}, 359(3):1079--1089, 2018.

\bibitem{MR2817648}
Y.~Ameur, H.~Hedenmalm, and N.~Makarov.
\newblock Fluctuations of eigenvalues of random normal matrices.
\newblock {\em Duke Math. J.}, 159(1):31--81, 2011.

\bibitem{MR3975882}
Y.~Ameur, N.-G. Kang, and N.~Makarov.
\newblock Rescaling {W}ard identities in the random normal matrix model.
\newblock {\em Constr. Approx.}, 50(1):63--127, 2019.

\bibitem{MR4030288}
Y.~Ameur, N.-G. Kang, N.~Makarov, and A.~Wennman.
\newblock Scaling limits of random normal matrix processes at singular boundary
  points.
\newblock {\em J. Funct. Anal.}, 278(3):108340, 46, 2020.

\bibitem{2016arXiv160908582B}
R.~{Bauerschmidt}, P.~{Bourgade}, M.~{Nikula}, and H.-T. {Yau}.
\newblock {The two-dimensional Coulomb plasma: quasi-free approximation and
  central limit theorem}.
\newblock {\em Adv. Theor. Math. Phys.}, 23(4):841--1002, 2019.

\bibitem{MR3903320}
R.~J. Berman.
\newblock Determinantal point processes and fermions on polarized complex
  manifolds: bulk universality.
\newblock In {\em Algebraic and analytic microlocal analysis}, volume 269 of
  {\em Springer Proc. Math. Stat.}, pages 341--393. Springer, Cham, 2018.

\bibitem{MR1794066}
P.~Bleher, B.~Shiffman, and S.~Zelditch.
\newblock Universality and scaling of correlations between zeros on complex
  manifolds.
\newblock {\em Invent. Math.}, 142(2):351--395, 2000.

\bibitem{2018arXiv181112225B}
R.~{Butez} and D.~{Garc{\'\i}a-Zelada}.
\newblock {Extremal particles of two-dimensional Coulomb gases and random
  polynomials on a positive background}.
\newblock arXiv:1811.12225.

\bibitem{MR4079754}
D.~Chafa\"{\i}, D.~Garc\'{\i}a-Zelada, and P.~Jung.
\newblock Macroscopic and edge behavior of a planar jellium.
\newblock {\em J. Math. Phys.}, 61(3):033304, 18, 2020.

\bibitem{MR3215627}
D.~Chafa{\"{\i}} and S.~P{\'e}ch{\'e}.
\newblock A note on the second order universality at the edge of coulomb gases
  on the plane.
\newblock {\em J. Stat. Phys.}, 156(2):368--383, 2014.

\bibitem{MR1643533}
L.-L. Chau and O.~Zaboronsky.
\newblock On the structure of correlation functions in the normal matrix model.
\newblock {\em Commun. Math. Phys.}, 196(1):203--247, 1998.

\bibitem{MR2411911}
T.~Claeys and A.~B.~J. Kuijlaars.
\newblock Universality in unitary random matrix ensembles when the soft edge
  meets the hard edge.
\newblock In {\em Integrable systems and random matrices}, volume 458 of {\em
  Contemp. Math.}, pages 265--279. Amer. Math. Soc., Providence, RI, 2008.

\bibitem{MR3664397}
F.~D. Cunden, P.~Facchi, M.~Ligab\`o, and P.~Vivo.
\newblock Universality of the third-order phase transition in the constrained
  {C}oulomb gas.
\newblock {\em J. Stat. Mech.: Theory Exp.}, (5):053303, 18, 2017.

\bibitem{MR2306224}
P.~Deift and D.~Gioev.
\newblock Universality at the edge of the spectrum for unitary, orthogonal, and
  symplectic ensembles of random matrices.
\newblock {\em Commun. Pure Appl. Math.}, 60(6):867--910, 2007.

\bibitem{MR2514781}
P.~Deift and D.~Gioev.
\newblock {\em Random matrix theory: invariant ensembles and universality},
  volume~18 of {\em Courant Lecture Notes in Mathematics}.
\newblock Courant Institute of Mathematical Sciences, New York; American
  Mathematical Society, Providence, RI, 2009.

\bibitem{MR2172690}
P.~Elbau and G.~Felder.
\newblock Density of eigenvalues of random normal matrices.
\newblock {\em Commun. Math. Phys.}, 259(2):433--450, 2005.

\bibitem{MR1236195}
P.~J. Forrester.
\newblock The spectrum edge of random matrix ensembles.
\newblock {\em Nuclear Phys. B}, 402(3):709--728, 1993.

\bibitem{MR2641363}
P.~J. Forrester.
\newblock {\em Log-gases and random matrices}, volume~34 of {\em London
  Mathematical Society Monographs Series}.
\newblock Princeton University Press, Princeton, NJ, 2010.

\bibitem{MR173726}
J.~Ginibre.
\newblock Statistical ensembles of complex, quaternion, and real matrices.
\newblock {\em J. Math. Phys.}, 6:440--449, 1965.

\bibitem{MR3056295}
H.~Hedenmalm and N.~Makarov.
\newblock Coulomb gas ensembles and {L}aplacian growth.
\newblock {\em Proc. Lond. Math. Soc. (3)}, 106(4):859--907, 2013.

\bibitem{2017arXiv171006493H}
H.~{Hedenmalm} and A.~{Wennman}.
\newblock {Planar orthogonal polynomials and boundary universality in the
  random normal matrix model}.
\newblock arXiv:1710.06493.

\bibitem{MR2552864}
J.~B. Hough, M.~Krishnapur, Y.~Peres, and B.~Vir\'{a}g.
\newblock {\em Zeros of {G}aussian analytic functions and determinantal point
  processes}, volume~51 of {\em University Lecture Series}.
\newblock American Mathematical Society, Providence, RI, 2009.

\bibitem{MR3761607}
B.~Lacroix-A-Chez-Toine, A.~Grabsch, S.~N. Majumdar, and G.~Schehr.
\newblock Extremes of 2d {C}oulomb gas: universal intermediate deviation
  regime.
\newblock {\em J. Stat. Mech.: Theory Exp.}, (1):013203, 39, 2018.

\bibitem{MR3788208}
T.~Lebl\'{e} and S.~Serfaty.
\newblock Fluctuations of two dimensional {C}oulomb gases.
\newblock {\em Geom. Funct. Anal.}, 28(2):443--508, 2018.

\bibitem{MR2129906}
M.~L. Mehta.
\newblock {\em Random matrices}, volume 142 of {\em Pure and Applied
  Mathematics (Amsterdam)}.
\newblock Elsevier/Academic Press, Amsterdam, third edition, 2004.

\bibitem{MR1606719}
D.~Petz and F.~Hiai.
\newblock Logarithmic energy as an entropy functional.
\newblock In {\em Advances in differential equations and mathematical physics
  ({A}tlanta, {GA}, 1997)}, volume 217 of {\em Contemp. Math.}, pages 205--221.
  Amer. Math. Soc., Providence, RI, 1998.

\bibitem{MR1986426}
B.~Rider.
\newblock A limit theorem at the edge of a non-{H}ermitian random matrix
  ensemble.
\newblock {\em J. Phys. A}, 36(12):3401--3409, 2003.
\newblock Random matrix theory.

\bibitem{MR2035641}
B.~Rider.
\newblock Order statistics and {G}inibre's ensembles.
\newblock {\em J. Stat. Phys.}, 114(3-4):1139--1148, 2004.

\bibitem{MR1485778}
E.~B. Saff and V.~Totik.
\newblock {\em Logarithmic potentials with external fields}, volume 316 of {\em
  Grundlehren der Mathematischen Wissenschaften [Fundamental Principles of
  Mathematical Sciences]}.
\newblock Springer-Verlag, Berlin, 1997.
\newblock Appendix B by Thomas Bloom.

\bibitem{MR655364}
E.~R. Smith.
\newblock Effects of surface charge on the two-dimensional one-component
  plasma. {I}. {S}ingle double layer structure.
\newblock {\em J. Phys. A}, 15(4):1271--1281, 1982.

\bibitem{MR1727234}
A.~Soshnikov.
\newblock Universality at the edge of the spectrum in {W}igner random matrices.
\newblock {\em Commun. Math. Phys.}, 207(3):697--733, 1999.

\bibitem{MR1257246}
C.~A. Tracy and H.~Widom.
\newblock Level-spacing distributions and the {A}iry kernel.
\newblock {\em Commun. Math. Phys.}, 159(1):151--174, 1994.

\bibitem{MR1385083}
C.~A. Tracy and H.~Widom.
\newblock On orthogonal and symplectic matrix ensembles.
\newblock {\em Commun. Math. Phys.}, 177(3):727--754, 1996.

\bibitem{MR1986427}
P.~Wiegmann and A.~Zabrodin.
\newblock Large scale correlations in normal non-{H}ermitian matrix ensembles.
\newblock volume~36, pages 3411--3424. 2003.
\newblock Random matrix theory.

\end{thebibliography}

\end{document}